\documentclass[12pt, a4paper, oneside]{article}
\usepackage[top=1in, bottom=1.1in, left=0.9in, right=0.5in]{geometry}
\usepackage[parfill]{parskip}
\usepackage{amsmath}
\usepackage{amsthm}
\usepackage{graphicx}
\usepackage{amssymb}
\usepackage{epstopdf}
\usepackage{setspace}
\usepackage{authblk}
\usepackage{enumerate}
\usepackage{lmodern}
\usepackage[hypcap]{caption}
\usepackage[english]{babel}
\usepackage{fancyhdr}
\addto\captionsenglish{}
\addto\captionsenglish{}
\addto\captionsenglish{}
\addto\captionsenglish{}
\addto\captionsenglish{}

\newlength\longest

\begin{document}
	
	\newtheorem{theorem}{Theorem}[section]
	\newtheorem{lemma}[theorem]{Lemma}
	\newtheorem{corollary}[theorem]{Corollary}
	\newtheorem{proposition}[theorem]{Proposition}
	\newtheorem{conjecture}[theorem]{Conjecture}
	\newtheorem{problem}[theorem]{Problem}
	\newtheorem{claim}[theorem]{Claim}
	\theoremstyle{definition}
	\newtheorem{assumption}[theorem]{Assumption}
	\newtheorem{fact}[theorem]{Fact}
	\newtheorem{remark}[theorem]{Remark}
	\newtheorem{definition}[theorem]{Definition}
	\newtheorem{example}[theorem]{Example}
	\theoremstyle{remark}
	\newtheorem{notation}{Notasi}
	\renewcommand{\thenotation}{}

\title{A Note on the Formula for the $g$-Angle between Two Subspaces of a Normed Space}
\author{M. Nur${}^{1}$\footnote{\emph{Permanent Address}:
		Department of Mathematics, Hasanuddin University,
		Jl. Perintis Kemerdekaan KM 10, Makassar 90245, Indonesia}, and H. Gunawan${}^{2}$ }
\affil{${}^{1,2}$Analysis and Geometry Group, Faculty of Mathematics and\\
	Natural Sciences, Bandung Institute of Technology,\\
	Jl. Ganesha 10, Bandung 40132, Indonesia\\
\bigskip
E-mail: ${}^{1}$muhammadnur@unhas.ac.id, ${}^{2}$hgunawan@math.itb.ac.id,
}
\date{}

\maketitle
\begin{abstract}
\noindent We introduce a new $2$-norm on a normed space using a semi-inner product $g$ on the
space. Using the $2$-norm, we propose a formula for the $g$-angle between $2$-dimensional subspaces
in the space. Our formula serves as a revision of the one proposed by Nur {\it et al.} \cite{Nur1}.

\bigskip{}

\bigskip{}
\end{abstract}

\section{Introduction}

In an inner product space $(X,\left\langle\cdot,\cdot\right\rangle)$, we can calculate the
angle $A(x,y)$ between two nonzero vectors $x$ and $y$ in $X$ via the formula
$A(x,y) :=\frac{\left\langle x,y\right\rangle}{\|x\| \| y\| },$
where $\| x\|:=\left\langle x,x\right\rangle^\frac{1}{2}$ denotes the induced norm in $X$.
In 2005, Gunawan {\it et al.} \cite{Gunawan4} presented a formula for the angle between an
$n$-dimensional subspace and an $m$-dimensional subspace of $X$ (with $m\ge n$) by using
the so-called standard $n$-norm on $X$.

Here we shall formulate the angle between 2-dimensional subspaces of a normed space, using
a semi-inner product on the space. Let $(X,\| \cdot \| )$ be a normed space.
The functional $g:X^{2}\rightarrow \mathbb{R}$ defined by the formula
$$
g(x,y):=\frac{1}{2}\left\Vert x\right\Vert \left[ \tau _{+}(x,y)+\tau _{-}(x,y)\right],
$$
with
\begin{equation*}
\tau _{\pm }(x,y):=\lim_{t \rightarrow \pm 0}\frac{\| x+ty\| -\| x\| }{t},
\end{equation*}
clearly satisfies the following properties:
\begin{enumerate}
\item[(1)] $g(x,x) =\| x\|^{2}$ for every $x\in X$;

\item[(2)] $g(ax,by)=ab\cdot g(x,y)$ for every $x,y\in X$ and $a,b\in \mathbb{R}$;

\item[(3)] $g(x,x+y)=\| x\|^{2}+g(x,y)$for every $x,y\in X$;

\item[(4)] $| g(x,y)| \leq \|x\|\cdot\|y\|$ for every $x,y\in X$.
\end{enumerate}

If, in addition, the functional $g(x,y)$ is linear in $y$, then $g$ is called
a \emph{semi-inner product} on $X$. For example, the functional
\begin{equation*}
g(x,y):=\|x\|_{p}^{2-p}\sum |\xi
_{k}|^{p-1}\text{sgn}\left( \xi _{k}\right) \eta _{k},\text{ \ \ }%
x:=\left( \xi _{k}\right), y:=\left( \eta _{k}\right) \in l^{p}
\end{equation*}
is a semi-inner product on $\ell^{p}\ (1\le p<\infty)$ \cite{Giles,Gunawan3}.
Note that on an inner product space, the functional $g(x,y)$ is identical
with the inner product $\left\langle\cdot,\cdot\right\rangle$. 

Using a semi-inner product $g$ on $X$, many researchers have studied the $g$-angle between two vectors, see,
for example \cite{Balestro,Gunawan3,Milicic1,Nur0}. Recently, Nur {\it et al.} \cite{Nur1}
formulated the $g$-angle between two subspaces of $X$. If $U=\text{span}\{u\}$ is a
$1$-dimensional subspace and $V=\text{span}\{v_{1},\cdots,v_{m}\}$ is an $m$-dimensional
subspace of $X$ with $m\ge 1$, then the $g$-angle between $U$ and $V$ is defined
by $A_{g}(U,V)$ with $\cos ^{2}A_{g}(U,V) =\frac{\|u_{V}\|^{2}}{\|u\|^{2}}$. In this formula,
$u_{V}$ denotes the $g$-orthogonal projection of $u$ on $V$. Likewise, if
$U=\text{span}\{u_{1},u_{2}\}$ is a $2$-dimensional subspace and $V=\text{span}\{v_{1},
\cdots,v_{m}\}$ is an $m$-dimensional subspace of $X$ with $m\ge 2$, then the $g$-angle
between $U$ and $V$ is defined by $\cos ^{2}A_{g}(U,V)=\frac{(\Lambda_{p} (u_{1V},u_{2V}))^{2}}
{(\Lambda_{p}(u_{1},u_{2}))^{2}}$, where $u_{iV}$'s denote the $g$-orthogonal projection of
$u_{i}$'s on $V$, with $i=1,2$. This formula, however, depends on the choice of the basis for
$U$, which is undesirable.

In this article, we will define a new $2$ norm on $X$ using a semi-inner product $g$.
Recall that a $2$-norm on a real vector space $X$ is a mapping
$\| \cdot,\cdot \| :X\times X\longrightarrow \mathbb{R}$ which satisfies the following four
conditions:
\begin{enumerate}
	\item[(1)] $\| x,y\| =0$ if and only if $x,y$ are linearly dependent;
	
	\item[(2)]  $\| x,y\|$ is invariant under permutation;
	
	\item[(3)]  $\|\alpha  x,y\| =|\alpha| \|  x,y|$ for every $x,y\in X$ and for every $\alpha \in\mathbb{R}$;
	
	\item[(4)]  $\|x,y+z\| \leq \|x,y\|+\|x,z\| $ for every $x,y,z\in X$.
\end{enumerate}
The pair $(X,\left\Vert \cdot,\cdot \right\Vert )$ is called a \emph{$2$-normed space}.
Geometrically, $\|x,y\|$ may be interpreted as the area of the 2-dimensional parallelepiped
spanned by $x$ and $y$. The theory of $2$-normed spaces was first developed by G\"{a}hler
\cite{Gahler1} in the mid 1960's. Recent results can be found, for example, in \cite{Ekariani1,Gunawan1,Gunawan2}.

Using a 2-norm, we will formulate the $g$-angle between two 2-dimensional subspaces of $X$, which serves as
a revision of Nur {\it et al.}'s formula.

\section{Main Results}

\subsection{A new $2$-norm} In this section, we will present a new $2$-norm on a real normed space $(X,\|\cdot\|)$.
(Unless otherwise stated, we shall always assume that $X$ is a normed space.)
Let $g(\cdot,\cdot)$ be a semi-inner product on $X$.
We define the mapping $\|\cdot,\cdot\|_{g}$ on $X$ by
\begin{equation}\label{2norm}
\left\Vert x_{1},x_{2}\right\Vert _{g}=\sup_{y_{j}\in X,\left\Vert
	y_{j}\right\Vert \leq 1}\left\vert
\begin{array}{cc}
g(y_{1},x_{1}) & g(y_{2},x_{1}) \\
g(y_{1},x_{2}) & g(y_{2},x_{2})%
\end{array}%
\right\vert.
\end{equation}

The following fact tells us that $\|\cdot,\cdot\|_{g}$ makes sense.

\bigskip

\begin{fact}\label{fact}
 	The inequality
 	\begin{equation*}
 	\|x_{1},x_{2}\|_{g}\le 2\|x_{1}\|\|x_{2}\|
 	\end{equation*}
 holds	for every $x_{1},x_{2}\in X$.
\end{fact}

\begin{proof}
Let $y_{1}, y_{2}\in X$ with $\|y_{1}\|\le 1$ and $\|y_{2}\|\le 1$. By the triangle inequality
for real numbers, we have
\begin{equation*}
g(y_{1},x_{1})g(y_{2},x_{2})-g(y_{2},x_{1})g(y_{1},x_{2})\le 2\|x_{1}\|\|x_{2}\|.
\end{equation*}
Hence, $\|x_{1},x_{2}\|_{g}\le 2\|x_{1}\|\|x_{2}\|$.
\end{proof}

Moreover, we have the following result.

\bigskip

\begin{proposition}\label{norma2}
The mapping (\ref{2norm}) defines a $2$-norm on $X$.
\end{proposition}

\begin{proof}
We need to check that $\|\cdot,\cdot\|_{g}$ satisfies the four properties of a 2-norm.	
\begin{enumerate}
	\item Let $x_{1}=kx_{2}$ with $k\in\mathbb{R}$. Observe that
	\begin{equation*}
	\left\vert
	\begin{array}{cc}
	g(y_{1},x_{1}) & g(y_{2},x_{1}) \\
	g(y_{1},x_{2}) & g(y_{2},x_{2})%
	\end{array}%
	\right\vert =0.
	\end{equation*}
	Hence, $\left\Vert x_{1},x_{2}\right\Vert _{g}=0$. 	Conversely, if
$\left\Vert x_{1},x_{2}\right\Vert _{g} =0$ then  the rows of the matrix
	\begin{equation*}
	\left[
	\begin{array}{cc}
	g(y_{1},x_{1}) & g(y_{2},x_{1}) \\
	g(y_{1},x_{2}) & g(y_{2},x_{2})%
	\end{array}%
	\right]
	\end{equation*}%
	are linearly dependent for all $y_{1},y_{2}$ with $\|y_{1}\|$, $\|y_{2}\|\le 1$.
    This happens only if $x_{1}$ dan $x_{2}$ are linearly dependent.
	\item By using the properties of determinants, we obtain
$\left\Vert x_{1},x_{2}\right\Vert _{g}=\left\Vert x_{2},x_{1}\right\Vert _{g}$.
	\item  Let $\alpha\in \mathbb{R}$. Using the properties of determinants and supremum, we obtain
		$\left\Vert \alpha x_{1},x_{2}\right\Vert _{g} =|\alpha| \left\Vert x_{1},x_{2}\right\Vert _{g}$.
	
	\item Observe that for arbitrary $x_{1},x_{2},z\in X$, we obtain
	\begin{eqnarray*}
		\left\Vert x_{1},x_{2}+z\right\Vert _{g} &=&\sup_{y_{j}\in X,\left\Vert
			y_{j}\right\Vert \leq 1}\left\vert
		\begin{array}{cc}
			g(y_{1},x_{1}) & g(y_{2},x_{1}) \\
			g(y_{1},x_{2}+z) & g(y_{2},x_{2}+z)%
		\end{array}%
		\right\vert  \\
		&=&\sup_{y_{j}\in X,\left\Vert y_{j}\right\Vert \leq 1}\left\{ \left\vert
		\begin{array}{cc}
			g(y_{1},x_{1}) & g(y_{2},x_{1}) \\
			g(y_{1},x_{2}) & g(y_{2},x_{2})%
		\end{array}%
		\right\vert +\left\vert
		\begin{array}{cc}(
			g(y_{1},x_{1}) & g(y_{2},x_{1}) \\
			g(y_{1},z) & g(y_{2},z)%
		\end{array}%
		\right\vert \right\}  \\
		&\leq &\sup_{y_{j}\in X,\left\Vert y_{j}\right\Vert \leq 1}\left\vert
		\begin{array}{cc}
			g(y_{1},x_{1}) & g(y_{2},x_{1}) \\
			g(y_{1},x_{2}) & g(y_{2},x_{2})%
		\end{array}%
		\right\vert +\sup_{y_{j}\in X,\left\Vert y_{j}\right\Vert \leq 1}\left\vert
		\begin{array}{cc}
			g(y_{1},x_{1}) & g(y_{2},x_{1}) \\
			g(y_{1},z) & g(y_{2},z)%
		\end{array}%
		\right\vert  \\
		&=&\left\Vert x_{1},x_{2}\right\Vert _{g}+\left\Vert x_{1},z\right\Vert _{g}
	\end{eqnarray*}
\end{enumerate}
This completes the proof.
\end{proof}

For an inner product space, we have the following fact.

\bigskip

\begin{fact}
	Let $(X,\left\langle\cdot,\cdot\right\rangle)$ be a real inner product space.
The two formulas $\|\cdot,\cdot\|_{g}$ in (\ref{2norm}) and $\|\cdot,\cdot\|_{s}$  with
	\begin{equation*}
	\|x_{1},x_{2}\|_{s} :=\left|
	\begin{array}{cc}
	\left\langle x_{1},x_{1}\right\rangle & \left\langle x_{1},x_{2}\right\rangle \\
	\left\langle x_{2},x_{1}\right\rangle & \left\langle x_{2},x_{2}\right\rangle%
	\end{array}%
	\right|^{\frac{1}{2}}
	\end{equation*}
for every $x_{1},x_{2}\in X$, are identical.
\end{fact}

\begin{proof}
On the inner product space $X$, the semi-inner product $g(\cdot,\cdot)$ is identical with the inner product
$\left\langle\cdot,\cdot\right\rangle$.  Therefore,
	\begin{equation*}
\left\Vert x_{1},x_{2}\right\Vert _{g}=\sup_{y_{j}\in X,\left\Vert
	y_{j}\right\Vert \leq 1}\left\vert
\begin{array}{cc}
	\left\langle y_{1},x_{1}\right\rangle & \left\langle y_{1},x_{2}\right\rangle \\
	\left\langle y_{2},x_{1}\right\rangle & \left\langle y_{2},x_{2}\right\rangle%
\end{array}%
\right\vert. 	
\end{equation*}
By applying the generalized Cauchy-Schwarz inequality \cite{Kurepa} and Hadamard's
inequality \cite{Gantmacher}, we obtain
\begin{align*}
\left\Vert x_{1},x_{2}\right\Vert _{g} \le \sup_{y_{j}\in X,\left\Vert
	y_{j}\right\Vert \leq 1} \|x_{1},x_{2}\|_{s}\|y_{1},y_{2}\|_{s}\le \|x_{1},x_{2}\|_{s}.
\end{align*}	
Conversely, assuming that $\{x_{1},x_{2}\}$ are linearly independent.
By using the Gram-Schmidt process, we have $\{x'_{1},x'_{2}\}$ are orthogonal. Moreover,
$\|x_{1},x_{2}\|_{s}=\|x'_{1},x'_{2}\|_{s}=\|x'_{1}\|\|x'_{2}\|$.
If $y_{1}=\frac{x'_{1}}{\|x'_{1}\|}$ and $y_{2}=\frac{x'_{2}}{\|x'_{2}\|}$,
then $\|y_{1}\|=1$ and $\|y_{2}\|=1$.
Next, using the properties of inner product and determinants, we obtain
\begin{eqnarray*}
	\left\vert
	\begin{array}{cc}
		\left\langle y_{1},x_{1}\right\rangle  & \left\langle
		y_{1},x_{2}\right\rangle  \\
		\left\langle y_{2},x_{1}\right\rangle  & \left\langle
		y_{2},x_{2}\right\rangle
	\end{array}%
	\right\vert  &=&\left\vert
	\begin{array}{cc}
		\left\langle y_{1},x_{1}^{\prime }\right\rangle  & \left\langle
		y_{1},x_{2}^{\prime }\right\rangle  \\
		\left\langle y_{2},x_{1}^{\prime }\right\rangle  & \left\langle
		y_{2},x_{2}^{\prime }\right\rangle
	\end{array}%
	\right\vert =\frac{1}{\left\Vert x_{1}^{\prime }\right\Vert \left\Vert
		x_{2}^{\prime }\right\Vert }\left\vert
	\begin{array}{cc}
		\left\langle x_{1}^{\prime },x_{1}^{\prime }\right\rangle  & \left\langle
		x_{1}^{\prime },x_{2}^{\prime }\right\rangle  \\
		\left\langle x_{2}^{\prime },x_{1}^{\prime }\right\rangle  & \left\langle
		x_{2}^{\prime },x_{2}^{\prime }\right\rangle
	\end{array}%
	\right\vert  \\
	&=&\left\Vert x_{1}^{\prime }\right\Vert \left\Vert x_{2}^{\prime
	}\right\Vert=\|x_{1},x_{2}\|_{s}.
\end{eqnarray*}
Thus, $\|x_{1},x_{2}\|_{g}\ge \|x_{1},x_{2}\|_{s}$, so that $\|x_{1},x_{2}\|_{g}= \|x_{1},x_{2}\|_{s}$.
Next, if $\{x_{1},x_{2}\}$ are linearly dependent, then $\|x_{1},x_{2}\|_{g}= \|x_{1},x_{2}\|_{s}=0$.
\end{proof}

Note that in an inner product space, we have a better inequality for Fact \ref{fact}, namely
$\|x_{1},x_{2}\|_{g}\le \|x_{1}\|\|x_{2}\|$. This is Hadamard's inequality for $n=2$ \cite{Gantmacher}.

\subsection{The $g$-angle between $2$-dimensional subspaces}
Here, using the $2$-norm $\|\cdot,\cdot\|_{g}$, we will formulate the $g$-angle between $2$-dimensional
subspaces of $X$. First, we recall the definition of the $g$-orthogonal projection
of $u$ on a subspace of $X$ as follows.

\bigskip

\begin{definition}\label{fact:2.3}
\cite{Milicic0} Let $u$ be a vector of $X$ and $S=\text{span}\{x_{1},\dots,x_{n}\}$ be
a subspace of $X$ with $\Gamma (x_{1},\dots,x_{n})=\det [g(x_{i},x_{k})]\neq 0$. (This additional
condition is added because it does not always follow from the linearly independence condition.)
The $g$-{\it orthogonal projection} of $u$ on $S$, denoted by $u_{S}$, is defined by
\begin{align*}
u_{S}:=-\frac{1}{\Gamma (x_{1},\dots, x_{n})}\left\vert
\begin{array}{cccc}
0 & x_{1} & \cdots  & x_{n} \\
g(x_{1},u) & g(x_{1},x_{1}) & \cdots  & g(x_{1},x_{n}) \\
\vdots  & \vdots  & \ddots  & \vdots  \\
g(x_{n},u) & g(x_{n},x_{1}) & \cdots  & g(x_{n},x_{n})%
\end{array}%
\right\vert,
\end{align*}
and its $g$-{\it orthogonal complement} $u-u_{S}$ is given by
\begin{align*}
u-u_{S}=\frac{1}{\Gamma (x_{1},\dots, x_{n})}\left\vert
\begin{array}{cccc}
u & x_{1} & \cdots  & x_{n} \\
g(x_{1},u) & g(x_{1},x_{1}) & \cdots  & g(x_{1},x_{n}) \\
\vdots  & \vdots  & \ddots  & \vdots  \\
g(x_{n},u) & g(x_{n},x_{1}) & \cdots  & g(x_{n},x_{n})%
\end{array}%
\right\vert.
\end{align*}
\end{definition}

Note that the notation of the determinant $|\cdot|$ here has a special meaning
because the elements of the matrix are not in the same field.

Let $U$ and $V$ be subspaces of $X$. Take arbitrary $u_{1},u_{2}\in U$ and
$\alpha,\beta\in\mathbb{R}$. Using properties of determinants and semi-inner product-$g$,
we have $(\alpha u_{1}+\beta u_{2})_{V}=\alpha u_{1V}+\beta u_{2V}$. Hence,
the $g$-orthogonal projection of $U$ on $V$ is a linear transformation from $U$ to $V$.

Next, let $x_{1},\dots,x_{n}\in X$ be a finite sequence of linearly independent
vectors. We can construct a {\it left $g$-orthonormal sequence} $x_{1}^{\ast
},\dots,x_{n}^{\ast }$ with $x_{1}^{\ast }:=\frac{x_{1}}{\|x_{1}\|}$ and
\begin{equation}
x_{k}^{\ast }:=\frac{x_{k}-\left(x_{k}\right)_{S_{k-1}}}
{\|x_{k}-\left(x_{k}\right)_{S_{k-1}}\|},
\end{equation}%
where $S_{k-1}=\text{span}\left\{x_{1}^{\ast },\dots,x_{k-1}^{\ast }\right\} $.
Note that ${\rm span}\{x_1^*,\dots,x_{k-1}^*\}={\rm span}\{x_1,\dots,x_{k-1}\}$
for each $k=2,\dots,n$, and that
$x_{k}^{\ast}\perp _{g}x_{l}^{\ast}$ for $k,l=1,\dots,n$ with $k<l$
(see \cite{Gunawan5,Milicic0}). We also observe that $\Gamma(x_1^*,\dots,x_{k-1}^*)=1$
for each $k=2,\dots,n$.

From the properties of the 2-norm and the $g$-orthogonal projection, we have the following lemma.

\bigskip

\begin{lemma}\label{Lemma1}
	If $U=\text{span}\{u_{1},u_{2}\}$ and
$V=\text{span}\{v_{1},v_{2}\}$ are $2$-dimensional subspaces of $X$ where $\{v_{1},v_{2}\}$
is left $g$-orthonormal, then
		\begin{equation*}
		\left\vert
		\begin{array}{cc}
		g(y_{1},u_{1V}) & g(y_{2},u_{1V}) \\
		g(y_{1},u_{2V}) & g(y_{2},u_{2V})%
		\end{array}%
		\right\vert =\left\vert
		\begin{array}{cc}
		g(v_{1},u_{1}) & g(v_{1},u_{2}) \\
		g(v_{2},u_{1}) & g(v_{2},u_{2})%
		\end{array}%
		\right\vert \left\vert
		\begin{array}{cc}
		g(y_{1},v_{1}) & g(y_{2},v_{1}) \\
		g(y_{1},v_{2}) & g(y_{2},v_{2})%
		\end{array}%
		\right\vert
		\end{equation*}
for every $y_{1},y_{2}\in X$.
\end{lemma}

\begin{proof}
If $\{v_{1},v_{2}\}$ is left $g$-orthonormal, then  $\Gamma\{v_{1},v_{2}\}=1$. Consequently, 	
\begin{equation*}
g(y_{j},u_{iV})=-\left\vert
\begin{array}{ccc}
0 & g\left( y_{j},v_{1}\right)  & g\left( y_{j},v_{2}\right)  \\
g\left( v_{1},u_{i}\right)  & 1 & 0 \\
g\left( v_{2},u_{i}\right)  & g\left( v_{2},v_{1}\right)  & 1%
\end{array}%
\right\vert
\end{equation*}
for $i,j=1,2$. By using the properties of determinants, we obtain
\small{
	\begin{eqnarray*}
		&&\left\vert
		\begin{array}{cc}
			g(y_{1},u_{1V}) & g(y_{2},u_{1V}) \\
			g(y_{1},u_{2V}) & g(y_{2},u_{2V})%
		\end{array}%
		\right\vert  \\
		&=&\left\vert
		\begin{array}{ccc}
			0 & g(y_{1},v_{1})  & g(y_{1},v_{2})  \\
			g(v_{1},u_{1})  & 1 & 0 \\
			g(v_{2},u_{1})  & g(v_{2},v_{1})  & 1%
		\end{array}%
		\right\vert \left\vert
		\begin{array}{ccc}
			0 & g(y_{2},v_{1})  & g(y_{2},v_{2})  \\
			g(v_{1},u_{2})  & 1 & 0 \\
			g(v_{2},u_{2})  & g( v_{2},v_{1})  & 1%
		\end{array}%
		\right\vert  \\
		&&-\left\vert
		\begin{array}{ccc}
			0 & g(y_{2},v_{1})  & g(y_{2},v_{2})  \\
			g(v_{1},u_{1})  & 1 & 0 \\
			g(v_{2},u_{1})  & g( v_{2},v_{1})  & 1%
		\end{array}%
		\right\vert \left\vert
		\begin{array}{ccc}
			0 & g(y_{1},v_{1})  & g(y_{1},v_{2})  \\
			g(v_{1},u_{2})  & 1 & 0 \\
			g(v_{2},u_{2})  & g(v_{2},v_{1})  & 1%
		\end{array}%
		\right\vert  \\
		&=&\left(-g( y_{1},v_{1}) g(v_{1},u_{1}) +g(y_{1},v_{2}) g(v_{1},u_{1}) g(v_{2},v_{1})
		-g(y_{1},v_{2}) g(v_{2},u_{1}) \right)  \\
		&&\left( -g(y_{2},v_{1}) g(v_{1},u_{2}) +g(y_{2},v_{2}) g(v_{1},u_{2}) g(v_{2},v_{1})
		-g(y_{2},v_{2}) g(v_{2},u_{2}) \right)  \\
		&&-\left(-g(y_{2},v_{1}) g(v_{1},u_{1}) +g(y_{2},v_{2}) g(v_{1},u_{1}) g(v_{2},v_{1})
		-g(y_{2},v_{2}) g(v_{2},u_{1}) \right)  \\
		&&\left( -g(y_{1},v_{1}) g(v_{1},u_{2}) +g(y_{1},v_{2}) g(v_{1},u_{2}) g(v_{2},v_{1})
		-g(y_{1},v_{2}) g(v_{2},u_{2}) \right)  \\
		&=&g(y_{1},v_{1}) g(y_{2},v_{1}) g(	v_{1},u_{1}) g(v_{1},u_{2}) -g( y_{1},v_{1})
		g(y_{2},v_{2}) g(v_{1},u_{1}) g(
		v_{1},u_{2}) g(v_{2},v_{1}) + \\
		&&g(y_{1},v_{1}) g(y_{2},v_{2}) g(v_{1},u_{1}) g(v_{2},u_{2}) -g(y_{1},v_{2})
		g(y_{2},v_{1}) g(v_{1},u_{1}) g(v_{1},u_{2}) g(v_{2},v_{1}) + \\
		&& g(y_{1},v_{2}) g(y_{2},v_{2}) g(v_{1},u_{1}) g(v_{1},u_{2}) g(v_{2},v_{1})
		^{2}-g(y_{1},v_{2}) g(y_{2},v_{2}) g(v_{1},u_{1}) g(v_{2},u_{2})g(v_{2},v_{1})   \\
		&&+g(y_{1},v_{2}) g(y_{2},v_{1}) g(v_{1},u_{2}) g(v_{2},u_{1}) -g(y_{1},v_{2}) g(y_{2},v_{2}) g(v_{1},u_{2}) g(v_{2},u_{1}) g( v_{2},v_{1})  \\
		&&+g(y_{1},v_{2})g(y_{2},v_{2})g(v_{2},u_{1})g(v_{2},u_{2}) -g(y_{1},v_{1}) g(y_{2},v_{1}) g(v_{1},u_{1}) g(v_{1},u_{2}) \\
		&& +g(y_{1},v_{2})g(y_{2},v_{1}) g(v_{1},u_{1}) g(v_{1},u_{2}) g(v_{2},v_{1})-g( y_{1},v_{2}) g( y_{2},v_{1}) g(v_{1},u_{1}) g( v_{2},u_{2})+\\
		&& g(y_{1},v_{1})g( y_{2},v_{2}) g( v_{1},u_{1}) g(v_{1},u_{2}) g (v_{2},v_{1})-g(y_{1},v_{2}) g( y_{2},v_{2}) g(v_{1},u_{1}) g(v_{1},u_{2}) g( v_{2},v_{1})^{2} \\
		&&+g(y_{1},v_{2}) g(y_{2},v_{2}) g(v_{1},u_{1}) g(v_{2},u_{2}) g(v_{2},v_{1})-g( y_{1},v_{1}) g(y_{2},v_{2}) g(v_{2},u_{1}) g( v_{1},u_{2})   \\
		&&+g(y_{1},v_{2})g( y_{2},v_{2}) g(v_{1},u_{2}) g(v_{2},u_{1} g( v_{2},v_{1})-g( y_{1},v_{2}) g(y_{2},v_{2}) g(v_{2},u_{1}) g(v_{2},u_{2})  \\
		&=&g(y_{1},v_{1}) g(y_{2},v_{2}) g(v_{1},u_{1}) g(v_{2},u_{2} +g( y_{1},v_{2})
		g(y_{2},v_{1}) g(v_{1},u_{2}) g(v_{2},u_{1})  \\
		&&-g(y_{1},v_{2}) g(y_{2},v_{1}) g(v_{1},u_{1}) g(v_{2},u_{2}) -g( y_{1},v_{1})
		g(y_{2},v_{2}) g(v_{2},u_{1}) g(v_{1},u_{2})  \\
		&=&g(v_{1},u_{1}) g( v_{2},u_{2}) \left\vert
		\begin{array}{cc}
			g(y_{1},v_{1}) & g(y_{2},v_{1}) \\
			g(y_{1},v_{2}) & g(y_{2},v_{2})%
		\end{array}%
		\right\vert -g(v_{1},u_{2}) g(v_{2},u_{1})
		\left\vert
		\begin{array}{cc}
			g(y_{1},v_{1}) & g(y_{2},v_{1}) \\
			g(y_{1},v_{2}) & g(y_{2},v_{2})%
		\end{array}%
		\right\vert.
	\end{eqnarray*}
}
Hence,
\begin{equation*}
\left\vert
\begin{array}{cc}
g(y_{1},u_{1V}) & g(y_{2},u_{1V}) \\
g(y_{1},u_{2V}) & g(y_{2},u_{2V})%
\end{array}%
\right\vert =\left\vert
\begin{array}{cc}
g(v_{1},u_{1}) & g(v_{1},u_{2}) \\
g(v_{2},u_{1}) & g(v_{2},u_{2})%
\end{array}%
\right\vert \left\vert
\begin{array}{cc}
g(y_{1},v_{1}) & g(y_{2},v_{1}) \\
g(y_{1},v_{2}) & g(y_{2},v_{2})%
\end{array}%
\right\vert.
\end{equation*}
This proves the lemma.
\end{proof}

Let us now define the $g$-angle between 2-dimensional subspace $U =\text{span}\{u_{1},u_{2}\}$
and 2-dimensional subspace $V$ of $X$ by
\begin{equation}
\cos ^{2}A_{g}(U,V):=\frac{\Vert u_{1V},u_{2V}\Vert _{g}^{2}}{\Vert
	u_{1},u_{2}\Vert _{g}^{2} \ \, \underset{{\rm span}\{v_1,v_2\}=V}{\sup }\Vert v_{1}^{\ast
	},v_{2}^{\ast }\Vert _{g}^{2}}   \label{4}
\end{equation}
where $u_{iV}$'s denote the the $g$-orthogonal projection of $u_{i}$'s on $V$
with $i=1,2$, and $\{v_{1}^{\ast },v_{2}^{\ast }\}$ is the left $g$-orthonormal
set obtained from $\{v_1,v_2\}$.

\bigskip

\begin{remark}
On an inner product space, the definition of the $g$-angle in (\ref{4}) is
identical with the angle defined in \cite{Gunawan4}, namely
\begin{equation*}
\cos ^{2}A_{g}(U,V)=\frac{\|u_{1V},u_{2V}\|_{g}^{2}}{%
	\|u_{1},u_{2}\|_{g}^{2}}
\end{equation*}.
\end{remark}

\bigskip

\begin{remark}
Let $\left\{ v_{1},v_{2}\right\} $ be a linearly independent set that spans $V$.
If $v_{1}^{\prime}=v_{1}$ and $v_{2}^{\prime }=v_{2}-\frac{g(v_{1},v_{2})}{\left\Vert
 	v_{1}\right\Vert }v_{1}$, then $\left\{ v_{1}^{^{\prime }},v_{2}^{\prime
 }\right\}$ is left $g$-orthogonal. Likewise, if $w_{1}^{\prime }=v_{2}$
 and $w_{2}^{\prime }=v_{1}-\frac{g(v_{2},v_{1})}{\left\Vert v_{2}\right\Vert
 }v_{2}$, then $\left\{ w_{1}^{\prime },w_{2}^{\prime }\right\} $
 is also left $g$-orthogonal. Using properties of the $2$-norm,
 we obtain $\|v_{1},v_{2}\|_{g}=\Vert v_{1}^{\prime },v_{2}^{\prime }\Vert
 _{g}=\Vert w_{1}^{\prime },w_{2}^{\prime
 }\Vert _{g}$. But, in general,
 $\Vert v_{1}^{\ast },v_{2}^{\ast }\Vert _{g}\neq \Vert w_{1}^{\ast },w_{2}^{\ast }\Vert _{g}$,
 where $v_{i}^{\ast }=\frac{v_{i}^{\prime }}{\left\Vert v_{i}^{\prime }\right\Vert }$ and
 $w_{i}^{\ast }=\frac{w_{i}^{\prime }}{\left\Vert w_{i}^{\prime }\right\Vert }$ with $i=1,2.$
 For instance, take $v_{1}=(1,0,0,\dots)$ and $v_{2}=(1,1,0,\dots)$ in $l^{1}$ with the usual
 semi-inner product $g$. If $v_{1}^{^{\prime }}=v_{1}$ and $v_{2}^{\prime
 }=v_{2}-\frac{g(v_{1},v_{2})}{\left\Vert v_{1}\right\Vert }v_{1}=(0,1,0,\dots)$,
 then $\left\Vert v_{1}^{\prime }\right\Vert \left\Vert v_{2}^{\prime
 }\right\Vert =1$. Next, if  $w_{1}^{^{\prime }}=v_{2}$ and  $%
 w_{2}^{\prime }=v_{1}-\frac{g(v_{2},v_{1})}{\left\Vert v_{2}\right\Vert }%
 v_{2}=(\frac{1}{2},-\frac{1}{2},0,\ldots )$, then  $\left\Vert
 w_{1}^{\prime }\right\Vert \left\Vert w_{2}^{\prime }\right\Vert =2.$ Hence
 $\Vert v_{1}^{\ast },v_{2}^{\ast }\Vert _{g}\neq $ $\Vert w_{1}^{\ast
 },w_{2}^{\ast }\Vert _{g}.$ Here we only change the order of the basis for $V$.
 Consider if we change the basis with another. This explains why we have the supremum term in the formula.
\end{remark}

According to the following theorem, the definition of the $g$-angle in (\ref{4}) makes sense.

\bigskip

\begin{theorem}\label{theosudut}
The ratio on the right hand side of (\ref{4}) is a number in $[0, 1]$ and
is independent of the choice of basis for $U$ and $V$ .
\end{theorem}

\begin{proof}
Let $\{v_{1},v_{2}\}$ be a linearly independent set that spans $V$.
Using the process in (2), we obtain the left
$g$-orthonormal set $\{v_{1}^{\ast },v_{2}^{\ast }\}$. Notice that
$\text{span}\{ v_{1},v_{2}\}=\text{span}\{v_{1}^{\ast },v_{2}^{\ast }\}$.
Using Lemma \ref{Lemma1} and Definition $\|\cdot,\cdot\|_{g}$ in (\ref{2norm}), we have
\begin{align*}
	\left\Vert u_{1V},u_{2V}\right\Vert _{g} =\left\vert
	\begin{array}{cc}
		g(v_{1}^{\ast },u_{1}) & g(v_{1}^{\ast },u_{2}) \\
		g(v_{2}^{\ast },u_{1}) & g(v_{2}^{\ast },u_{2})%
	\end{array}%
	\right\vert \left\Vert v_{1}^{\ast },v_{2}^{\ast }\right\Vert _{g}
	\end{align*}%
Since $\left\Vert v_{i}^{\ast }\right\Vert =1$ for $i=1,2$, we have%
\begin{align*}
	\left\Vert u_{1V},u_{2V}\right\Vert _{g} &=\left\vert
		\begin{array}{cc}
			g(v_{1}^{\ast },u_{1}) & g(v_{1}^{\ast },u_{2}) \\
			g(v_{2}^{\ast },u_{1}) & g(v_{2}^{\ast },u_{2})%
		\end{array}%
		\right\vert \left\Vert v_{1}^{\ast },v_{2}^{\ast }\right\Vert _{g} \\
	&\leq \sup_{y_{j}\in X,\left\Vert y_{j}\right\Vert \leq 1}\left\vert
		\begin{array}{cc}
			g(y_{1},u_{1}) & g(y_{2},u_{1}) \\
			g(y_{1},u_{2}) & g(y_{2},u_{2})%
		\end{array}%
		\right\vert \left\Vert v_{1}^{\ast },v_{2}^{\ast }\right\Vert _{g} \\
	&=\left\Vert u_{1},u_{2}\right\Vert _{g}\left\Vert v_{1}^{\ast
		},v_{2}^{\ast }\right\Vert _{g}  \\
	&\le  \left\Vert u_{1},u_{2}\right\Vert
	_{g}\underset{{\rm span}\{w_1,w_2\}=V}{\sup }\left\Vert w_{1}^{\ast },w_{2}^{\ast
	}\right\Vert_{g}.
\end{align*}
so that $$\frac{\left\Vert u_{1V},u_{2V}\right\Vert _{g}}{\left\Vert
	u_{1},u_{2}\right\Vert _{g}\underset{{\rm span}\{w_1,w_2\}=V}{\sup }\left\Vert w_{1}^{\ast
	},w_{2}^{\ast }\right\Vert_{g}}\le 1.$$

\noindent Secondly, note that the $g$-orthogonal projection of $u_{i}$'s on $V$ is independent
of the choice of basis for $V$ \cite{Milicic0}. Moreover, since the $g$-orthogonal projection of 
$U$ on $V$ is a linear transformation from $U$ to $V$, the ratio of (\ref{4}) is also invariant under every change of basis
for $U$. Indeed, the ratio is unchanged if swap $u_{1}$ and $u_{2}$, replace $u_{1}$ with
$u_{1}+\alpha u_{2}$, replace $u_{1}$ with $\alpha u_{1}$ or  $u_{2}$ with $\alpha u_{2}$ where $\alpha\neq 0$. The proof is complete.
\end{proof}

\section{Concluding Remarks}

The formula (\ref{4}) can be used to compute the $g$-angle between two subspaces of $\ell^p$ as follows.
Let $\{ v_{1},\dots,v_{m}\} $ be a linearly independent set that spans $V$ in $\ell^p$.
Using the process in (2), we obtain the left $g$-orthonormal set $\{v_{1}^{\ast },\dots,v_{m}^{\ast }\}$.
Here $\text{span}~\{v_1,\ldots,v_m\}=\text{span}~\{v_{1}^{\ast },\dots,v_{m}^{\ast }\}$.
Hence, for $i=1,2$,
$$
	u_{iV} =\,-\left\vert
	\begin{array}{cccc}
		0 & v_{1}^{\ast } & \cdots  & v_{m}^{\ast } \\
		g(v_{1}^{\ast },u_{i})  & g(v_{1}^{\ast },v_{1}^{\ast
		})  & \cdots  & g(v_{1}^{\ast },v_{m}^{\ast })  \\
		\vdots  & \vdots  & \ddots  & \vdots  \\
		g(v_{m}^{\ast },u_{i})  & g(v_{m}^{\ast },v_{1}^{\ast
		})  & \cdots  & g(v_{m}^{\ast },v_{m}^{\ast })
	\end{array}
	\right\vert
	=\,-\left\vert
	\begin{array}{cccc}
		g(v_{1}^{\ast },v_{1}^{\ast })  & \cdots  & g(v_{m}^{\ast
		},v_{1}^{\ast })  & v_{1}^{\ast } \\
		\vdots  & \ddots  & \vdots  & \vdots  \\
		g(v_{1}^{\ast },v_{m}^{\ast })  & \cdots  & g(v_{m}^{\ast
		},v_{m}^{\ast })  & v_{m}^{\ast } \\
		g(v_{1}^{\ast },u_{i})  & \cdots  & g(v_{m}^{\ast
		},u_{i})  & 0%
	\end{array}
	\right\vert.
$$
Substituting $g(v_{k}^{\ast },v_{i}^{\ast }) =\| v_{k}^{\ast
}\|_{p}^{2-p}\sum\limits_{j_{k}}|v_{kj_{k}}^{\ast}|
^{p-1}\text{sgn}(v_{kj_{k}}^{\ast })v_{ij_{k}}^{\ast }$, we obtain
\begin{align*}
u_{iV}=-\sum\limits_{j_{m}}\cdots \sum\limits_{j_{1}}|v_{1j_{m}}^{\ast
}|^{p-1}sgn(v_{1j_{m}}^{\ast })\cdots |v_{mj_{1}}^{\ast
}|^{p-1}sgn(v_{mj_{1}}^{\ast })\left\vert
\begin{array}{cccc}
v_{1j_{1}}^{\ast } & \cdots  & v_{1j_{m}}^{\ast } & v_{1}^{\ast } \\
\vdots  & \ddots  & \vdots  & \vdots  \\
v_{mj_{1}}^{\ast } & \cdots  & v_{mj_{m}}^{\ast } & v_{m}^{\ast } \\
u_{ij_{1}} & \cdots  & u_{ij_{m}} & 0%
\end{array}%
\right\vert.
\end{align*}
\bigskip
Using this formula, we can compute the value of the $g$-angle between two subspaces $U =\text{span}\{u_{1},u_{2}\}$
and $V =\text{span}\{v_{1},v_{2}\}$ of $\ell^{p}$ for $1\le p<\infty$. For instance, in $\ell
^{2}$, let $U=\text{span}\{u_{1},u_{2}\}$ and $V=\text{span}\{v_{1},v_{2}\}$
with $u_{1}=(1,1,2,0,\dots )$, $u_{2}=(2,1,3,0,\dots )$, 
$v_{1}=(1,0,0,0,\dots )$, and $v_{2}=(0,1,0,0,\dots )$. We obtain $%
u_{1V}=\left( 1,1,0,0,\dots \right) $ and $u_{2V}=(2,1,0,0,0,\dots )$.
Moreover, $\Vert u_{1}\Vert =\sqrt{6}$, $\Vert u_{2}\Vert =\sqrt{14},$ $%
\Vert u_{1V}\Vert =\sqrt{2}$, and $\Vert u_{2V}\Vert =\sqrt{5}$. Observe
that $\sup\limits_{{\rm span}\{w_1,w_2\}=V} \|w_1^*,w_2^*\|_g=\left\Vert w_1^*,w_2^*\right\Vert_{s}=1.$ Next%
\begin{equation*}
\left\Vert u_{1V},u_{2V}\right\Vert _{g}=\left\Vert u_{1V},u_{2V}\right\Vert
_{s}=\sqrt{10-9}=1
\end{equation*}%
and%
\begin{equation*}
\left\Vert u_{1},u_{2}\right\Vert _{g}=\left\Vert u_{1},u_{2}\right\Vert
_{s}=\sqrt{84-81}=\sqrt{3}.
\end{equation*}%
Thus $\cos ^{2}A_{g}(U,V)=\frac{1}{3}$, so that $A_{g}(U,V)=\arccos (\frac{1%
}{3}\sqrt{3})$.

\bigskip

\textbf{Acknowledgement}. The research is supported by ITB Research and
Innovation Program 2018.


\begin{thebibliography}{99}

\bibitem{Balestro} V. Balestro,  \`A.G. Horv\`ath, H. Martini, and R. Teixeira,
``Angles in normed spaces'', \emph{Aequationes Math.}
{\bf 91}--2 (2017), 201--236.

\bibitem{Ekariani1} S. Ekariani, H. Gunawan, and M. Idris, ``A contractive mapping
theorem on the $n$-normed space of $p$-summable'', \emph{JMA.} {\bf 4}--1 (2013), 1--7.

\bibitem{Gahler1}  S. G\"{a}hler, ``Lineare $2$-normierte R\"{a}ume'',
\emph{Math. Nachr.} {\bf 28} (1964), 1--43.

\bibitem{Giles} J. R. Giles, ``Classes of semi-inner-product spaces'',
 \emph{Trans. Amer. Math. Soc.} {\bf 129}--3 (1967), 436--446.

\bibitem{Gantmacher} F.R. Gantmacher, \emph{The Theory of Matrices},
AMS Chelsea Publishing {\bf Vol. 1} 2000.

\bibitem{Gunawan1} H. Gunawan and M. Mashadi, ``On finite-dimensional
$2$-normed space'', \emph{Soochow J. Math.} {\bf 27} (2001), 321--329.

\bibitem{Gunawan2} H. Gunawan and Mashadi, ``On $n$-normed spaces'',
\emph{Int. J. Math. Math. Sci.} {\bf 27} (2001), 631--639.

\bibitem{Gunawan3} H. Gunawan, J. Lindiarni, and O. Neswan,
``$P$-, $I$-, $g$-, and $D$-angles in normed spaces'',
\emph{J. Math. Fund. Sci.} {\bf 40}--1 (2008), 24--32.

\bibitem{Gunawan4} H. Gunawan, O. Neswan, and W. Setya-Budhi,
``A formula for angles between two subspaces of inner product spaces'',
\emph{Beitr. Algebra Geom.} {\bf 46}--2 (2005), 311--320.

\bibitem{Gunawan5} H. Gunawan, W. Setya-Budhi, S. Mashadi, and S. Gemawati,
``On volumes of $n$-dimensional parallelepipeds in $\ell^{p}$ spaces'',
\emph{Univ. Beograd Publ. Elektrotehn. Fak. Ser. Mat.} {\bf 16} (2005), 48--54.

\bibitem{Kurepa} S. Kurepa, ``On the Buniakowsky-Cauchy-Schwarz inequality'',
\emph{Glas. Mat. III Ser.} {\bf 21}--1 (1966), 147--158.

\bibitem{Milicic0} P.M. Mili\`ci\`c, ``On the Gram-Schmidt projection in normed spaces'',
\emph{Univ. Beograd Publ. Elektrotehn. Fak. Ser. Mat.} {\bf 4} (1993), 89--96.

\bibitem{Milicic1} P.M. Mili\`ci\`c, ``On the $B$-angle and $g$-angle in normed spaces'',
\emph{J. Inequal. Pure and Appl. Math.} {\bf 8}--3 (2007), 1--9.

\bibitem{Nur0} M. Nur, and H. Gunawan, ``A new orthogonality and angle in a normed space'',
\emph{Aequationes Math.} (2018), https://doi.org/10.1007/s00010-018-0582-3.

\bibitem{Nur1} M. Nur, H. Gunawan, and O. Neswan, ``A formula for the $g$-angle
between two subspaces of a normed space'', \emph{Beitr. Algebra Geom.} {\bf 59}-1 (2018),
133--143.

\end{thebibliography}
\end{document}